\documentclass[12pt,twoside]{article}
\usepackage{amsmath,amsfonts,amssymb}

\usepackage{amsbsy}
\usepackage{amssymb}
\usepackage{amsfonts}
\usepackage{amsmath}
\usepackage{amsopn}
\usepackage{amsthm}
\usepackage[english]{babel}
\usepackage[applemac]{inputenc}
\usepackage[T1]{fontenc}

\newtheorem{teo}{Theorem}
\newtheorem{lemma}{Lemma}

\newtheorem{defi}{Definition}

\theoremstyle{definition}
\newtheorem{remark}{Remark}

\newcommand{\gracias}{\noindent\textbf{Acknowledgement.}\ }

\DeclareMathOperator{\zr}{\mathbb{Z}}
\DeclareMathOperator{\re}{\mathbb{R}}

\DeclareMathOperator{\er}{\mathbf{E}}
\DeclareMathOperator{\pr}{\mathbf{P}}

\title{Tail asymptotics for exponential functionals of {L}\'evy processes: the convolution equivalent case\footnote{The main result in this paper has been announced in the workshop on infinitely divisible processes, March 16-20, 2009, held at CIMAT, Guanajuato, Mexico.}}
\author{V{\'\i }ctor Rivero\thanks{Centro de
Investigaci\'on en Matem\'aticas (CIMAT A.C.). Calle Jalisco s/n, 36240 Guanajuato, Guanajuato M\'exico. E-mail:
rivero@cimat.mx}}
\date{\empty}

\begin{document}
\maketitle
\begin{abstract}
We determine the rate of decrease of the right tail distribution of the exponential functional of a L\'evy process with a convolution equivalent L\'evy measure. Our main result establishes that it decreases as the right tail of the image under the exponential function of the L\'evy measure of the underlying L\'evy process. The method of proof relies on fluctuation theory of L\'evy processes and an explicit path-wise representation of the exponential functional as the exponential functional of a bivariate subordinator. Our techniques allow us to establish rather general estimates of the measure of the excursions out from zero for the underlying L\'evy process reflected in its past infimum, whose area under the exponential of the excursion path exceed a given value. 
\end{abstract}


\section{Introduction and main results}

Let $(\xi, \pr)$ be a real valued L\'evy process with characteristic triple $(a,\sigma,\Pi),$  $a\in\re,$ $\sigma\in\re$ and $\Pi$ a measure over $\re\setminus\{0\},$ such that $\int_{\re\setminus\{0\}}(x^{2}\wedge 1)\Pi(dx)<\infty.$ The characteristic exponent of $\xi,$ will be denoted by $\Psi,$ 
$$\er\left(e^{i\lambda\xi_{1}}\right)=\exp\left\{-\Psi(\lambda)\right\}=\exp\left\{-\left(ia\lambda+\frac{\lambda^{2}\sigma^{2}}{2}+\int_{\re\setminus\{0\}}\left(1-e^{i\lambda x}+i\lambda x1_{\{|x|<1\}}\right)\Pi(dx)\right)\right\},$$ for $\lambda\in \re.$ We assume that $\xi$ drifts to $-\infty,$ $\lim_{t\to\infty}\xi_{t}=-\infty,$ $\pr$-a.s. Later we will also assume that $\xi$ has some positive jumps, $\Pi(x,\infty)>0,$ $\forall x>0,$ and hence the case where $\xi$ is the negative of a subordinator is automatically excluded. For background on L\'evy processes see e.g. \cite{Be96}, \cite{doneybook} and \cite{kypbook}. 

In this paper we are interested by the asymptotic behavior of the right tail distribution of the \textit{exponential functional associated to the L\'evy process} $\xi,$ $$I:=\int^\infty_{0}e^{\xi_{s}}ds.$$  The assumption that $\xi$ drifts towards $-\infty$ implies that $I<\infty,$ $\pr$-a.s. see e.g. \cite{bysurvey}, Theorem 1.  

The exponential functionals of L\'evy processes have been the subject of several recent researches and had found a number of applications in  branching processes, composition structures, generalized Ornstein-Uhlenbeck processes, finance, financial time series, population dynamics, random algorithms, random media, risk theory, self-similar fragmentation theory, self-similar Markov processes theory, to name a few. The thorough review by Bertoin and Yor \cite{bysurvey}, is an excellent source of information about exponential functionals of L\'evy processes and their applications.    

Finding the distribution of the random variable $I$ is in general a difficult problem mainly because it is defined in terms of the whole path of the underlying L\'evy process. In fact, in the literature about the topic we can only find a few cases  where the law of $I$ is explicitly known, most of which  are enlisted in \cite{bysurvey}. Carmona, Petit and Yor \cite{CPY1997} proved that for a large class of L\'evy processes the law of $I$ admits a smooth density and that it solves an integral-differential equation. It can be seen in the discussion and examples in \cite{CPY1997} that solving such equation is  a hard task, even for L\'evy processes whose characteristics admit a simple form. Very recently, Patie~\cite{patie2009b} announced a formula  for the distribution of $I$ under the assumption that the underlying L\'evy process has no-positive jumps. Patie's formula involves an power series whose coefficients are given in terms of the Laplace exponent of $\xi$.  

As it often happens, in many applications it is enough to have estimates of the right tail distribution of $I,$  $t\mapsto \pr(I>t),$ as $t\to\infty.$ But by the same reasons described above these are not easy to obtain and there is no standard technique to attack the problem. This is a topic that has been studied by Haas \cite{haas03} and Rivero \cite{R03} in the case where the underlying L\'evy process is the negative of a subordinator; by Maulik and Zwart~\cite{MZ2006} under the assumption that $\xi$ is not the negative of a subordinator and a that the law of $\xi_{1}$ satisfies a condition of subexponentiality;  and by Rivero \cite{R05} under the assumption that $\xi$ is not the negative of a subordinator and $\xi$ satisfies the so-called Cram\'er condition and an integrability condition. As the former and latter articles were one of the motivations of this research and to locate our results in the right context we will next describe in more detail the results in \cite{MZ2006} and \cite{R05}.

In the paper \cite{MZ2006} it is assumed that $\xi$ satisfies that 
\begin{enumerate}
\item[(MZ1)] $\displaystyle \mu=-\er(\xi_{1})\in(0,\infty),$
\item[(MZ2)] $\displaystyle \overline{G}(x):=\min\left\{1, \int^{\infty}_{x}\pr(\xi_{1}>u)du\right\}$ is subexponential, or equivalently that $$\displaystyle \min\left\{1, \int^{\infty}_{x}\Pi(u,\infty)du\right\}$$ is subexponential. (For the definition of subexponential see the forthcoming Definition~\ref{def1}.)
\end{enumerate}
Under the assumptions (MZ1-2) Maulik and Zwart, in their Theorem 4.1, proved that 
\begin{equation}\label{MZ}
\pr(I>t)\sim\frac{1}{\mu}\overline{G}(\log(t)),\qquad t\to\infty.
\end{equation}
Besides, the hypotheses in the paper \cite{R05} are that
\begin{enumerate}
\item[(R1)] $\xi$ is not-arithmetic, that is, its state space is not a subgroup of $k\zr$ for any real number $k;$
\item[(R2)] The \textit{Cram\'er condition} is satisfied, that is, there exists a $\theta>0$ such that $\er(e^{\theta\xi_{1}})=1$;
\item[(R3)] $\er(\xi^{+}_{1}e^{\theta\xi_{1}})<\infty,$ with $a^{+}=\max\{a,0\}.$
\end{enumerate} The conditions (R2-3) are satisfied whenever $\xi$ has no positive jumps. Under the assumptions (R1-3) the author proved that  
\begin{equation}\label{R}
\pr(I>t)\sim Ct^{-\theta},\qquad t\to\infty,
\end{equation}
where $C\in(0,\infty)$ is a constant. The identity $C=\frac{\er(I^{\theta-1})}{\er(\xi_{1}e^{\theta\xi_{1}})}$ has been obtained in \cite{R05} under the assumption that $0<\theta<1,$ and later extended in \cite{MZ2006} to any $\theta>0$ under the assumption that $-\infty<\er(\xi_{1})<0$. A further formula for the constant $C$ has been obtained by Patie in~\cite{patie2009} for L\'evy processes with no-positive jumps and which satisfy some assumptions.

An intuition for the latter and former estimates rely on the conjecture that  
\begin{equation}\label{conjecture}
\pr(\log(I)>t)\sim c\pr(\sup_{s>0}\xi_{s}>t),\qquad t\to\infty,
\end{equation} for some constant $c\in(0,\infty).$
Which in turn is based on the heuristic that the large values of $I$ are due to the large values of $\exp\{\sup_{s>0}\xi_{s}\}.$ It can be verified that the latter conjecture holds true under the assumptions in \cite{MZ2006} and \cite{R05}. Namely, in \cite{MZ2006} it has been proved directly that 
\begin{equation*}
\pr(\log(I)>t)\sim\pr(\sup_{s>0}\xi_{s}>t)\sim\frac{1}{\mu}\overline{G}(t),\qquad t\to\infty.
\end{equation*}
Whereas, under the assumptions (R1-3) Bertoin and Doney \cite{BD94} proved that there exists a constant $c^{\prime}\in(0,\infty)$ such that $$\lim_{t\to\infty}e^{\theta t}\pr(\sup_{s>0}\xi_{s}>t)=c^{\prime}.$$ Hence, the latter estimate together with that one in (\ref{R}) confirms that under the assumptions (R1-3)  the conjecture (\ref{conjecture}) is verified. 

Our main purpose in this paper is twofold. First, providing an estimate for the right hand tail distribution of $I$ for a large class of L\'evy processes that do not satisfy the hypotheses (MZ1-2) nor (R1-3), namely that of L\'evy processes with a \textit{convolution equivalent} L\'evy measure.  Second, exhibiting other cases where the conjecture (\ref{conjecture}) is verified.

Before stating our main result we need to recall some basic notions. 

\begin{defi}\label{def1}A distribution function $G(x)<1$ for all $x\in\re,$ is said to be \textit{convolution equivalent} or \textit{close to exponential} if
\begin{enumerate}
\item[(a)] it has an exponential tail with rate $\gamma\geq 0,$ written $G\in\mathcal{L}_{\gamma},$ viz. $$\lim_{x\to\infty}\frac{\overline{G}(x-y)}{\overline{G}(x)}=e^{\gamma y},\qquad y\in\re,\qquad \overline{G}(x):=1-G(x), \quad x\in \re;$$
\item[(b)] and the following limit exists
$$\lim_{x\to\infty}\frac{\overline{G^{*2}}(x)}{\overline{G}(x)}:=2M<\infty,$$ where as usual $G^{*2}$ means $G$ convoluted with itself twice.
\end{enumerate}
In that case, we use the notation $G\in \mathcal{S}_{\gamma}.$ If $\gamma=0,$ the family $\mathcal{S}_{0}$ is better known as the class of subexponential distributions. It is known that $M=M_{G}:=\int_{\re}e^{\gamma x}dG(x),$ and that if $\gamma>0$ the convergence in (a) holds uniformly over intervals of the form $(b,\infty)$ for $b\in \re.$ 
\end{defi}
A result by Pakes \cite{pakes2004, pakes2007}, see also \cite{watanabe}, estabishes that if $F$ is an infinitely divisible distribution with L\'evy measure $\nu$ and $$J_{\nu}(x):=\frac{\nu[x\vee 1,\infty)}{\nu[1,\infty)},\qquad x\in\re.$$ Then we have the following equivalences $$J_{\nu}\in\mathcal{S}_{\gamma}\ \Longleftrightarrow\ J_{\nu}\in\mathcal{L}_{\gamma}\ \text{and} \ \lim_{x\to\infty}\frac{\overline{F}(x)}{\nu[x,\infty)}=M_{F}\ \Longleftrightarrow\ F\in\mathcal{S}_{\gamma},$$ where by $J_{\nu}\in\mathcal{S}_{\gamma}$ we mean the distribution whose right tail equals $J_{\nu}$ belongs to $\mathcal{S}_{\nu}.$ We will use the notation $\nu\in\mathcal{S}_{\alpha}$ whenever $J_{\nu}\in\mathcal{S}_{\alpha}.$ For further background on convolution equivalent distributions we refer to \cite{pakes2004, pakes2007}, \cite{watanabe} and the reference therein.

We will say that a L\'evy process is in $\mathcal{S}_{\gamma},$ for some $\gamma\geq 0,$ if the law of $\xi_{1}$ is in $\mathcal{S}_{\gamma}$ or equivalently its L\'evy measure $\Pi\in\mathcal{S}_{\gamma}.$ The class of L\'evy processes having this property has been deeply studied by Kyprianou, Kluppelberg \& Maller~\cite{KKM}. In Subsection~\ref{sec:prelim} we will recall some of their results, but we quote the following result here, as together with our main result will imply that the conjecture (\ref{conjecture}) is verified. 

\begin{lemma}[Kyprianou, Kluppelberg \& Maller~\cite{KKM}]\label{lem1}
Assume that $\xi$ is not arithmetic, that $\xi\in\mathcal{S}_{\alpha}$ for some $\alpha>0,$ and $\er(e^{\alpha\xi_{1}})<1.$  We have that 
$$\displaystyle\pr(\sup_{s\geq0} \xi_{s}>t)\sim \frac{\phi_{h}(0)}{(\phi_{h}(-\alpha))^{2}}\Pi(t,\infty),\qquad t\to\infty,$$ where $\phi_{h}$ denotes the Laplace exponent of the upward ladder height subordinator of $\xi.$
\end{lemma}

We have all the elements to state our main result.

\begin{teo}\label{teo1}
Assume that $\xi$ is not arithmetic, that $\xi\in\mathcal{S}_{\alpha}$ for some $\alpha>0,$ and $\er(e^{\alpha\xi_{1}})<1.$ If $0<\alpha\leq1,$ we assume furthermore that $\er(\xi_{1})\in(-\infty,0).$ We have that  
$$\displaystyle \pr(I>t)\sim \frac{\er(I^{\alpha})}{-\psi(\alpha)}\Pi(\log(t),\infty),\qquad t\to\infty,$$  where $\er(I^{\alpha})<\infty,$ $\psi(\alpha)=\log(\er(e^{\alpha\xi_{1}})).$  As a consequence the distribution of $\log(I)$ belongs to the class $\mathcal{S}_{\alpha},$ and the estimate (\ref{conjecture}) holds.
\end{teo}

The proof of this result will be given in the next two sections via a few Lemmas. Our arguments rely on the fluctuation theory for L\'evy processes and hence we will start the next section by introducing some notation and recalling some facts from this theory. A key ingredient in our approach is a path wise representation of the exponential functional $I$ as the exponential functional of a bivariate subordinator, $(\widehat{h},Y)$, that we characterize explicitly, viz. $I=\int^{\infty}_{0}e^{-h_{t-}}dY_{t}$. Such is the content of Lemma~\ref{lemma:1} and its proof uses excursion theory for the process $\xi$ reflected in its past infimum. The L\'evy measure of the subordinator $Y$ is the image of the excursion measure of the process $\xi$ reflected in its past infimum under the mapping that associates to an excursion path its area under the exponential function. Roughly speaking, the forthcoming  Lemma~\ref{grey} will be used in the proof of Theorem~\ref{teo1} to justify that the large values of $I$ come from a large jump of $Y,$ in the sense explained in \cite{HL07}, and therefore it will be crucial to have estimates of the right tail of the L\'evy measure of $Y.$  These will be obtained by means of the next result which can be seen as an analogue of Theorem~\ref{teo1} but for the exponential functional of the excursion of a L\'evy process reflected in its past infimum. 

\begin{teo}\label{teo2}
Assume that the hypotheses of Theorem~\ref{teo1} hold.  Let $\underline{n}$ be the measure of the excursions out from $0$ for the L\'evy process $\xi$ reflected in its past infimum, $\varepsilon$ be the canonical process and $\zeta$ its lifetime. The following tail estimate
$$\lim_{y \to \infty}\frac{\underline{n}\left(\int^{\zeta}_{0}e^{\varepsilon(t)}dt>y\right)}{{\Pi}(\log(y),\infty)}=\frac{\er(I^{\alpha})}{\phi_{h}(-\alpha)}\in(0,\infty),$$ holds, where $\phi_{h}$ denotes the Laplace exponent of the upward ladder height subordinator of $\xi.$ 
\end{teo}
We think that this result is of interest in itself as it reflects how the excursions of the reflected process mimics the behaviour of the whole process, and because in general little is known about the excursion measure of a L\'evy process reflected. This result is in the same vein as the one in \cite{doneymaller05} and as in that paper this may lead to a Poisson limit theorem for the number of excursion paths with large area under the exponential of the excursion, we leave the details to the interested reader.

The proof of Theorem~\ref{teo2} will be given by means of the Lemma~\ref{lemma:5}. Its proof is quite long and technical, so we consecrate most of Section~\ref{sec:3} to it. 

The fact that the estimate in Theorem~\ref{teo2} holds raises the question of whether the analogous result holds under the assumptions in \cite{MZ2006} or \cite{R05}. The following result answers this question.  
\begin{teo}\label{teo3} With the same notation as in Theorem~\ref{teo2}.
\begin{enumerate}
\item[(i)] Assume that the hypotheses (MZ1-2) hold. We have that
$$\lim_{y\to\infty}\frac{\underline{n}\left(\int^{\zeta}_{0}e^{\varepsilon(t)}dt>y\right)}{{\int^{\infty}_{\log(y)}\Pi}(x,\infty)dx}=0.$$
\item[(ii)] Assume that the hypotheses (R1-3) hold. We have that 
$$\lim_{y\to\infty}y^{\theta}\underline{n}\left(\int^{\zeta}_{0}e^{\varepsilon(t)}dt>y\right)=\frac{\er(I^{\theta-1})}{\mu^{(\theta)}_{h}},$$ where $\mu^{(\theta)}_{h}=\er(h_{1}e^{\theta h_{1}})\in(0,\infty)$ and $h$ denotes the upward ladder height process associated to $\xi.$
\end{enumerate}
\end{teo}

The proof of Theorem~\ref{teo3} will be given in Section~\ref{sect4}. 

It is important to observe that the arguments in the proof of Theorem~\ref{teo1} apply, using instead of the Lemma~\ref{lemma:5} the result in Theorem~\ref{teo3}-(ii), to give an alternative proof to the result in \cite{R05}, described in (\ref{R}) above. 

We finish this introduction by mentioning that an interesting related problem is determining the behaviour of the distribution $t\mapsto \pr(I\leq t),$ as $t\to 0+.$ This has been studied by Pardo~\cite{P2006} in the case where the underlying L\'evy process has no positive jumps and its Laplace exponent is regularly varying at infinity with index $\beta\in(1,2)$, and by Caballero and Rivero \cite{cr09} in the case where the underlying L\'evy process is the negative of a subordinator whose Laplace exponent is regularly varying at $0$. For the best of our knowledge there is no known conjecture or heuristic that allows to intuit the rate of decrease of $\pr(I\leq t),$ for general L\'evy processes.
\section{Some preliminaries} \label{sec:prelim}
We will first recall a few facts from fluctuation theory of L\'evy process, and we refer to \cite{Be96}, \cite{doneybook} and \cite{kypbook}  for further background on the topic. In the fluctuation theory of L\'evy processes it is well known that the process $\xi$ reflected in its past infimum
$$\xi_{t}-i_{t}:=\xi_{t}-\inf_{s\leq t}\xi_{s},\qquad t\geq 0,$$ is a strong Markov process in the filtration $(\mathcal{F}_{t})_{t\geq 0},$ the $\pr$-completed filtration generated by $\xi.$ So it admits a local time at $0,$ which we will denote by $\widehat{L}_{t}, t\geq 0.$ Let $\widehat{L}^{-1}$ be the downward ladder time process associated to $\xi,$ that is the right-continuous inverse of the local time $\widehat{L},$ and $\widehat{h}$ the downward ladder height process associated to $\xi,$ that is $$\widehat{h}_{t}=-i_{\widehat{L}^{-1}_{t}},\quad t\geq 0.$$ It is well known that for each $t>0,$ $\widehat{L}^{-1}_{t}$ is a $\mathcal{F}_{s}$-stopping time and $\xi_{\widehat{L}^{-1}_{t}}=-\widehat{h}_{t}$. The couple  $(\widehat{L}^{-1}_{t},\widehat{h}_{t}), t\geq 0$ is the so-called  downward ladder process associated to $\xi.$ We denote by $\underline{n}$ the measure of the excursions out from $0$ of the process $\xi$ reflected in its past infimum, and by $\varepsilon$ the coordinate process under $\underline{n}.$ We consider also the upward ladder process, $({L}^{-1}_{t}, {h}_{t}), t\geq 0,$ that is the downward ladder process associated to the dual L\'evy process $-\xi.$  We will denote by $\phi_{h}$ (respectively, $\phi_{\widehat{h}}$) the Laplace exponent of the upward (respectively, downward) ladder height subordinator $h,$ and by $(\kappa,a,\Pi_{h}),$ respectively $(\widehat{\kappa},\widehat{a},\Pi_{\widehat{h}}),$ its associated killing rate, drift and L\'evy measure. We define the tail L\'evy measure of $h,$ respectively $\widehat{h},$ by $$\overline{\Pi}_{h}(x)=\Pi_{h}(x,\infty), \quad \overline{\Pi}_{\widehat{h}}(x)=\Pi_{\widehat{h}}(x,\infty),\qquad x>0.$$ Furthermore, we denote by $V_{h}(dx),$ (respectively, $V_{\widehat{h}}(dx)$) the potential measure of $h,$ (respectively, $\widehat{h}$), viz. $$V_{h}(dx)=\int^{\infty}_{0}dt\pr(h_{t}\in dx),\quad V_{\widehat{h}}(dx)=\int^{\infty}_{0}dt\pr(\widehat{h}_{t}\in dx)\qquad x\geq 0.$$ The Wiener-Hopf factorization in space for $\xi$ tells us that there exists a constant $k^{\prime},$ whose value depends of the normalization of the local times, $\widehat{L}$ and $L,$ such that characteristic exponent of $\xi,$ $\Psi,$ can be factorized as
$$k^{\prime}\Psi(\lambda)=\phi_{h}(-i\lambda)\phi_{\widehat{h}}(i\lambda), \qquad \lambda\in\re.$$ We assume  without loss of generality that the local times are normalized so that $k^{\prime}=1.$  Let $C=\{\lambda\in\re: \er\left(e^{\lambda\xi_{1}}\right)<\infty\}.$  The characteristic exponent $\Psi$ can be extended by analytical continuation to the complex strip $-\Im(z)\in C.$ Thus we can define the Laplace exponent of $\xi$ by $$\er(e^{\lambda\xi_{1}})=e^{\psi(\lambda)},\quad \psi(\lambda)=-\Psi(-i\lambda),\quad\lambda\in C.$$ By Hold\"er's inequality the function $\psi$ is convex on $C,$ and so if $\psi(\lambda_{0})<0$ for some $\lambda_{0}\in C\cap (0,\infty),$ then $\psi(\lambda)<0$ for $0<\lambda<\lambda_{0}.$
It holds also that $$\er(e^{\lambda h_{1}})<\infty, \qquad \text{for}\ \lambda\in C\cap\re^{+},$$ and by analytical continuation  
$$\er(e^{\lambda h_{1}})=e^{-\phi_{h}(-\lambda)},\qquad \lambda\in C\cap\re^{+}.$$
The Wiener-Hopf factorization can be analytically extended to $C$ as \begin{equation}\label{WHextd}
\psi(\lambda)=(-\phi_{h}(-\lambda))\phi_{\widehat{h}}(\lambda),\qquad \lambda\in C,
\end{equation}
 see \cite{vigonthesis} chapters 4 and 6, for further details.

We next recall a few consequences of the assumption that $\xi$ is not arithmetic, that the law of $\xi_{1}\in\mathcal{S}_{\alpha},$ for some $\alpha>0,$ and $\er(e^{\alpha\xi_{1}})<1.$ The latter and former conditions are equivalent to $h\in\mathcal{S}_{\alpha},$ and $\er\left(e^{\alpha h_{1}}\right)<1.$ When this holds we have that
\begin{equation}\label{eq:tailequiv}
\Pi(x,\infty)=:\overline{\Pi}^{+}(x)\sim {\phi}_{\widehat{h}}(\alpha)\overline{\Pi}_{h}(x),\qquad x\to\infty.\end{equation} We have that $[0,\alpha]\subseteq C,$ $C\cap(\alpha,\infty)=\emptyset,$ and  that
\begin{equation}\label{eq:extdWH}\psi(\alpha)=(-\phi_{h}(-\alpha))\phi_{\widehat{h}}(\alpha),\qquad 0<\phi_{h}(-\alpha)<\infty.\end{equation} Furthermore, as $\xi$ drifts to $-\infty,$ it follows that $h$ has a finite lifetime, or equivalently $\phi_{h}(0)>0,$ and its renewal measure is a finite measure such  that 
\begin{equation}\label{renewalasym}\lim_{x\to\infty}\frac{\overline{V}(x)}{\overline{\Pi}_{h}(x)}=\frac{1}{(\phi_{h}(-\alpha))^{2}},\qquad \overline{V}_{h}(x):=V_{h}(x,\infty),\qquad x\geq 0.\end{equation}
see \cite{KKM} for a proof of these facts and other interesting related results. We have also that for any $y\in \re,$ 
\begin{equation}\label{U}
\lim_{x\to\infty}\frac{\overline{V}(x+h)}{\overline{\Pi}_{h}(x)}=\frac{e^{-\alpha h}}{(\phi_{h}(-\alpha))^{2}},\qquad \text{uniformly in}\ h\in(y,\infty).\end{equation}   

A key ingredient for our approach is the following representation result. 
\begin{lemma}\label{lemma:1}
\begin{enumerate} 
\item[(i)] The process $Y,$ defined by 
$$Y_{t}:=at+\sum_{u\leq t}\int^{\widehat{L}^{-1}_{u}-\widehat{L}^{-1}_{u-}}_{0}\exp\{(\xi_{s+\widehat{L}^{-1}_{u-}}-\xi_{\widehat{L}^{-1}_{u-}})\}\mathrm{d}s,\quad t\geq 0,$$ is a subordinator  
with drift $a$ determined by $$a\widehat{L}_{s}=\int^s_{0}1_{\{\xi_{u}=i_{u}\}}du,$$ and L\'evy measure $\Pi_{Y},$ given by $$\Pi_{Y(y,\infty)}=:\overline{\Pi}_{Y}(y)=\underline{n}\left(\int^{\zeta}_{0}e^{\varepsilon(u)}du>y\right),\qquad y>0.$$ The process $(\widehat{h},Y)$ is a bivariate L\'evy process. 
\item[(ii)] The path wise equality of processes 
$$\left(\int^{\widehat{L}^{-1}_{t}}_{0}e^{\xi_{u}}du, t\geq 0\right)=\left(\int^{t}_{0}e^{-\widehat{h}_{u-}}dY_{u}, t\geq 0\right),$$ holds.
\item[(iii)] The Laplace exponent of $Y$ can be represented as $$\frac{\phi_{Y}(\lambda)}{\lambda}=a+\int_{(0,\infty)}V_{h}(dx)e^{x}\er\left(\exp\left\{-\lambda e^{x}\int^{T_{(-\infty,-x)}}_{0}e^{\xi_{s}}ds\right\}\right),\qquad \lambda>0,$$ or equivalently the tail L\'evy measure of $Y$ is given by $$\int^y_{0}\overline{\Pi}_{Y}(u)du=\int_{(0,\infty)}V_{h}(dx)e^{x}\pr\left(e^{x}\int^{T_{(-\infty,-x)}}_{0}e^{\xi_{s}}ds\leq y\right),\qquad y\geq 0.$$
\end{enumerate}
\end{lemma}
\begin{proof}
The proof of the claims in (i) and (ii) use the same arguments as in the proof of Lemma 2 in \cite{CHKPR}, so we omit the details.
The claim in (iii) is a consequence of the fact that the Laplace exponent $\phi_{Y}$ can be written as follows 
\begin{equation*}
\begin{split}
\phi_{Y}(\lambda)&=a\lambda+\underline{n}\left(1-\exp\left\{{-\lambda}\int^{\zeta}_{0}e^{\varepsilon(u)}du\right\}\right)\\
&=a\lambda+\lambda\underline{n}\left(\int^{\zeta}_{0}dse^{\varepsilon(s)}\exp\left\{-\lambda\int^{\zeta}_{s}e^{\varepsilon(u)}du\right\}\right)\\
&=a\lambda+\lambda\underline{n}\left(\int^{\zeta}_{0}dse^{\varepsilon(s)}\left(\exp\left\{-\lambda\int^{\zeta}_{0}e^{\varepsilon(u)}du\right\}\right)\circ\theta_{s}\right).
\end{split}
\end{equation*}
Hence using the Markov property under $\underline{n}$ and the fact that under $\underline{n}$ the canonical process has the same law as $\xi$ killed at its first hitting time of $(-\infty,0],$ we obtain the equality.
\begin{equation*}
\begin{split}
\phi_{Y}(\lambda)&=a\lambda+\lambda\underline{n}\left(\int^{\zeta}_{0}dse^{\varepsilon(s)}\er_{\varepsilon(s)}\left(\exp\left\{-\lambda\int^{T_{(-\infty,0)}}_{0}e^{\xi_{u}}du\right\}\right)\right)
\end{split}
\end{equation*}
We conclude using that the renewal measure of the upward ladder height subordinator $h$ equals the occupation measure under $\underline{n},$ viz. $$V_{h}(dx)=a\delta_{\{0\}}(dx)+\underline{n}\left(\int^\zeta_{0}1_{\{\varepsilon(s)\in dx\}}ds\right),$$ see e.g. \cite{Be96} exercise VI.5, and making an integration by parts. Indeed, we have the equalities which are a consequence of Fubini's theorem 
\begin{equation*}
\begin{split}
\frac{\phi_{Y}(\lambda)}{\lambda}-a&=\int_{(0,\infty)}V_{h}(dx)e^{x}\er_{x}\left(\exp\left\{-\lambda\int^{T_{(-\infty,0)}}_{0}e^{\xi_{u}}du\right\}\right)\\
&=\int_{(0,\infty)}V_{h}(dx)e^{x}\er\left(\exp\left\{-\lambda e^{x}\int^{T_{(-\infty,-x)}}_{0}e^{\xi_{u}}du\right\}\right)\\
&=\lambda \int^{\infty}_{0}dze^{-\lambda z}\int_{(0,\infty)}V_{h}(dx)e^{x}\pr\left(e^{x}\int^{T_{(-\infty,-x)}}_{0}e^{\xi_{u}}du\leq z\right);
\end{split}
\end{equation*}
and by an integration by parts
\begin{equation*}
\begin{split}
\frac{\phi_{Y}(\lambda)}{\lambda}-a&=\int^{\infty}_{0}e^{-\lambda y}\overline{\Pi}_{Y}(y)dy=\lambda\int^{\infty}_{0}e^{-\lambda z}\left(\int^{z}_{0}\overline{\Pi}_{Y}(y)dy\right)dz,
\end{split}
\end{equation*}
which are valid for $\lambda>0.$ The claim follows by the uniqueness of the Laplace transform.
\end{proof}

\begin{remark}
Observe that a side consequence of the latter result is that $$I=\int^{\infty}_{0}e^{-\widehat{h}_{t-}}dY_{t}.$$ In particular, in the case where $\xi$ has no negative jumps we know that $\widehat{h}_{t}=ct$ for some $0<c<\infty,$ and therefore $I=\int^{\infty}_{0}e^{-ct}dY_{t}.$ Which implies that in this case $I$ is a self-decomposable random variable. This is a fact that has been observed by a number of authors using a completely different argument, it can be found e.g. in \cite{R03} page 468. By a classical result by Wolfe~\cite{wolfe} and Sato and Yamazato \cite{SY1,SY2}, it is known that there exists a subordinator $Z$ such that $\er(\log(1+Z_{1}))<\infty$ and $I\stackrel{\text{Law}}{=}\int^{\infty}_{0}e^{-s}dZ_{s}$. With the latter Lemma we have given a step further ahead by giving a path wise construction of the subordinator $Y,$ which has the same law as $Z_{\cdot c},$  and by describing its drift and L\'evy measure. 
\end{remark}

In the following Lemma we gather some useful results for exponential functionals of L\'evy processes, which will allow us to ensure that the constants appearing in our main results are finite and strictly positive. 

\begin{lemma}\label{lemma:30}
We have that for $\gamma>0,$ $$\er\left(I^{\gamma}\right)<\infty \quad \text{if and only if}\quad \er(e^{\gamma\xi_{1}})<1.$$ In that case, we have the identity
$$\er(I^{\gamma})=\frac{\gamma}{-\psi(\gamma)}\er(I^{\gamma-1}).$$
Besides, if $\mu=-\er(\xi_{1})\in(0,\infty)$ then $\er(I^{-1})=\mu.$
\end{lemma}
\begin{proof}
If $\gamma\in(0,1)$ the assertion of the Lemma has been proved in Theorem~1 in \cite{R07}. To prove the result in the case $\gamma\geq 1,$ we observe that the following formula holds
\begin{equation}\label{eq:7.0}
\er(I^{\gamma})=\frac{\gamma}{-\psi(\gamma)}\er(I^{\gamma-1}),\qquad  \gamma\in\{\lambda>0:\er(e^{\lambda\xi_{1}})<1\}.
\end{equation}
This is a well known formula, for a proof see e.g. \cite{CPY1997} Proposition 3.1 or the proof of the Lemma 2 in \cite{R07}. (Note that in the formula (7) in \cite{R07} a negative sign is missing, the correct formula is the one in (\ref{eq:7.0}) above; the proof of this formula in the op.cit. paper is correct.) With this formula at hand we have that if $\gamma\geq 1$ and $\er(e^{\gamma\xi_{1}})<1,$ then by iteration 
\begin{equation}\label{eq:8.0}
\er(I^{\gamma})=\begin{cases} \prod^{\gamma}_{k=1}\left(\frac{k}{-\psi(k)}\right), &\text{if}\ \gamma\in\{1,2,\ldots\},\\
\er(I^{\gamma-\lfloor\gamma\rfloor})\prod^{\lfloor\gamma\rfloor-1}_{k=0}\left(\frac{(\gamma-k)}{-\psi(\gamma-k)}\right), &\text{if}\ \gamma\notin\{1,2,\ldots\},
\end{cases}
\end{equation}
where $\lfloor\gamma\rfloor$ denotes the integer part of $\gamma.$ If $\gamma\in\{1,2,\ldots\}$ the proof of the implication $``\Leftarrow''$ follows from this equation. For $\gamma\notin\{1,2,\ldots\}$ it follows from the latter formula and the fact that the claim in the Lemma holds for $0<\gamma-\lfloor\gamma\rfloor<1,$ that if $\er(e^{\gamma\xi_{1}})<1,$ then $\er(I^{\gamma})<\infty.$ The proof of the reciprocal follows as in its counterpart in Lemma~2 in \cite{R07}, as the assumption that $\gamma<1$ is not used in that part of the proof. The final assertion can be found in~\cite{CPY1997} Proposition 3.1.
\end{proof}

\section{Proof of Theorems~\ref{teo1} and~\ref{teo2}}\label{sec:3}
An elementary observation that will be very useful in the sequel is that if $G\in\mathcal{L}_{\alpha}$ then $x\mapsto \overline{G}(\log(x)),$ $x>0,$ is a regularly varying function with index $-\alpha.$ As the assumptions of Theorem \ref{teo1} imply that $\overline{\Pi}^{+}(\log(\cdot))$ is a regularly varying  at infinity with index $-\alpha,$ it is natural to start the proof of that result by establishing conditions under which the law of $I$ has a tail distribution which is regularly varying at infinity. That is the purpose of the following Lemma.    

\begin{lemma}\label{grey}
Let $t>0$ fixed, and $Q:=\int^{\widehat{L}^{-1}_{t}}_{0}\exp\{\xi_{s}\}ds,$ $M:=e^{-\widehat{h}_{t}}.$  We have that the tail distribution of $I$ is regularly varying with index $-\beta,$ for some $\beta>0,$ if and only if the tail distribution of $Q$ does. In that case the estimate
\begin{equation*}
\pr(I>s)\sim \frac{1}{1-\er(M^{\beta})}\pr(Q>s),\qquad s\to\infty,
\end{equation*}
holds.
\end{lemma}
\begin{proof}
Applying the strong Markov property at the stopping time $\widehat{L}^{-1}_{t}$ we obtain the identity 
$$I=\int^{\widehat{L}^{-1}_{t}}_{0}\exp\{\xi_{s}\}ds+e^{\xi_{\widehat{L}^{-1}_{t}}}\int^{\infty}_{0}\exp\{\widetilde{\xi_{u}}\}du=Q+M\widetilde{I},$$ where  $\widetilde{\xi}_{s}=\xi_{\widehat{L}^{-1}_{t}+s}-\xi_{\widehat{L}^{-1}_{t}},$ $s\geq 0,$ and hence $\widetilde{I}$ has the same law as $I$ and it is independent of $\mathcal{F}_{\widehat{L}^{-1}_{t}}.$ Hence, the random variable $I$ satisfies the random recurrence equation, with $(Q,M)$ as above
$$I\stackrel{\text{Law}}{=}Q+M\widetilde{I},\quad (Q,M)\ \text{independent of } \widetilde{I}\stackrel{\text{Law}}{=}I.$$ As $M\leq 1,$ it has moments of all positive orders, and thus the claim in the Lemma follows from a simple application of the main result in \cite{grey94}. \end{proof}
We deduce from this Lemma that to reach our end it will be enough to prove that under the assumptions of Theorem~\ref{teo1}, $s\mapsto\pr(Q>s),$ $s>0,$ has the same rate of decrease as $\overline{\Pi}^{+}(\log(\cdot))$ at infinity. To reach that end we will need the following Lemma.

\begin{lemma}\label{lemma:5}
Assume that the hypotheses of Theorem~\ref{teo1} hold.  The following tail estimate
$$\lim_{y \to \infty}\frac{\underline{n}\left(\int^{\zeta}_{0}e^{\varepsilon(t)}dt>y\right)}{\overline{\Pi}_{h}(\log(y))}=\frac{\alpha}{(\phi_{h}(-\alpha))^2}\er(I^{\alpha-1}),$$ holds. The tail distribution of $Y_{1}$ is regularly varying with index $-\alpha.$
\end{lemma}

The proof of Theorem~\ref{teo2} is a straightforward consequence of Lemma~\ref{lemma:5} as we have seen in (\ref{eq:tailequiv}) that $\overline{\Pi}_{h}$ is asymptotically equivalent to $\overline{\Pi}^{+}$. The expression for the constant follow from the Wiener-Hopf factorization~(\ref{eq:extdWH}). The proof of this Lemma is rather long and technical because different ranges of values of $\alpha$ need different approaches, so we prefer to postpone its proof and proceed to the proof of Theorem~\ref{teo1}. 

\begin{proof}[Proof of Theorem \ref{teo1}]
A consequence of Lemma \ref{lemma:1}, is that $Q$ is the stochastic integral $Q=\int^{t}_{0}e^{-\widehat{h}_{s-}}dY_{s},$ with $t$ fixed. Where by Lemma \ref{lemma:5}, we can ensure that the tail distribution of $Y$ is regularly varying at infinity, and furthermore $e^{-\widehat{h}_{s-}},$ $s\geq 0$, is a bounded and predictable process because it is adapted and left-continuous, with respect to the filtration $\left(\mathcal{F}_{\widehat{L}^{-1}_{s}}, s\geq 0\right)$. Hence we have all the elements to apply the Theorem 3.4 in \cite{HL07} to ensure that 
$$\pr(Q>y)\sim \int^{t}_{0}ds e^{-s\phi_{\widehat{h}(\alpha)}}\pr(Y_{1}>y)\sim\frac{1-e^{-t\phi_{\widehat{h}}(\alpha)}}{\phi_{\widehat{h}}(\alpha)}\overline{\Pi}_{Y}(y),\qquad y\to\infty.$$ By Lemma \ref{grey} we have therefore that
$$\pr(I>y)\sim \frac{1}{1-\er\left(e^{-\alpha\widehat{h}_{1}}\right)}\pr(Q>y)\sim \frac{1}{\phi_{\widehat{h}}(\alpha)}\overline{\Pi}_{Y}(y),\qquad y\to\infty.$$ Then by Lemma \ref{lemma:5} and the estimate (\ref{eq:tailequiv}) we get the estimate 
$$\pr(I>y)\sim \frac{\alpha\er\left(I^{\alpha-1}\right)}{\phi_{\widehat{h}}(\alpha)\left(\phi_{h}(-\alpha)\right)^{2}}\overline{\Pi}_{h}(\log(y))\sim \frac{\alpha\er\left(I^{\alpha-1}\right)}{\left(\phi_{\widehat{h}}(\alpha)\phi_{h}(-\alpha)\right)^{2}}\overline{\Pi}^{+}(\log(y)),\qquad y\to\infty.$$
But by the extended Wiener-Hopf identity (\ref{eq:extdWH}) we have that $\phi_{\widehat{h}}(\alpha)\phi_{h}(-\alpha)=-\psi(\alpha)$ and by Lemma \ref{lemma:30} we have that $\alpha\er\left(I^{\alpha-1}\right)=-\psi(\alpha)\er\left(I^{\alpha}\right).$ From where the form of the constant follows. Finally the assertion that the law of $\log(I)$ is in $\mathcal{S}_{\alpha}$ follows from the tail equivalence property of convolution equivalent distribution, see Lemma 2.4 and Corollary 2.1 in \cite{pakes2004} or Lemma 2.1 in \cite{watanabe}.
\end{proof}

We proceed now to the proof of Lemma~\ref{lemma:5}. It will be sufficient to prove that the estimate therein holds, as it implies that the tail L\'evy measure of $Y$ is regularly varying at infinity. A classical result by Embrechts and Goldie ~\cite{EG} ensures that the latter is equivalent to the regular variation at infinity of the tail distribution of $Y_{1}$ and that in that case $$\pr(Y_{1}>y)\sim \overline{\Pi}_{Y}(y),\qquad y\to\infty.$$ 
\begin{proof}[Proof of Lemma~\ref{lemma:5}: case $\alpha<1$] By Karamata's Tauberian theorem, Corollary 8.1.7 in \cite{BGT}, it is well known that the regular variation at infinity with index $\alpha<1,$ of the function $y\mapsto\overline{\Pi}_{Y}(y)$ is equivalent to the regular variation at $0$ of $\phi_{Y},$ and in that case
$$\lim_{u\to\infty}\frac{\overline{\Pi}_{Y}(u)}{\phi_{Y}(1/u)}=\frac{1}{\Gamma(1-\alpha)},$$ see e.g. \cite{Be96} page 75. So to prove the claimed result it will be enough to prove that $$\lim_{\lambda\to 0}\frac{\phi_{Y}(\lambda)}{\overline{\Pi}_{h}(\log(1/\lambda))}=\frac{\alpha\Gamma(1-\alpha)}{(\phi_{h}(-\alpha))^{2}}\er\left(I^{\alpha-1}\right).$$ 
Using the representation of $\phi_{Y}$ in Lemma~\ref{lemma:1} we deduce that 
\begin{equation*}
\begin{split}
\frac{\phi_{Y}(\lambda)}{\overline{\Pi}_{h}(\log(1/\lambda))}&=\frac{a\lambda}{\overline{\Pi}_{h}(\log(1/\lambda))}+\frac{\lambda}{\overline{\Pi}_{h}(\log(1/\lambda))}\int_{(0,\infty)} V_{h}(dx)e^{x}\er\left(\exp\left\{-\lambda e^{x}S_{x}\right\}\right)\\
&=\frac{a\lambda}{\overline{\Pi}_{h}(\log(1/\lambda))}+\frac{\lambda}{\overline{\Pi}_{h}(\log(1/\lambda))}\int_{(\log(\beta/\lambda), \infty)} V_{h}(dx)e^{x}\er\left(\exp\left\{-\lambda e^{x}S_{x}\right\}\right)\\
&\qquad +\frac{\lambda}{\overline{\Pi}_{h}(\log(1/\lambda))}\int_{(0, \log(\beta/\lambda))} V_{h}(dx)e^{x}\er\left(\exp\left\{-\lambda e^{x}S_{x}\right\}\right)\\
&=: \frac{a\lambda}{\overline{\Pi}_{h}(\log(1/\lambda))}+A_{1}(\lambda)+A_{2}(\lambda),
\end{split}
\end{equation*} for $\beta>0,$ and 
with $S_{x}:=\int^{T_{(-\infty,-x)}}_{0}e^{\xi_{s}}ds,$ $x>0.$ It follows from the fact that $t\mapsto \overline{\Pi}_{h}(\log(t))$ is regularly varying at infinity with an index $-\alpha,$ $\alpha\in(0,1),$ that $$\frac{a\lambda}{\overline{\Pi}_{h}(\log(1/\lambda))}\xrightarrow[\lambda\to 0]{}0.$$ It remains to study the behaviour of $A_{1}$ and $A_{2.}$ Observe that by construction the process $x\mapsto S_{x}$ is increasing, hence \begin{equation}\label{Sincreasing} 
S_{\log(r\beta)}\leq S_{\log(\beta/\lambda)}\leq S_{x}\uparrow_{x\to\infty}I,\ \text{for any} \ x>\log(\beta/\lambda),\ 1>r\lambda, \lambda>0.\end{equation}
Using the uniformity property in (\ref{U})  and the latter upper bound we can estimate $A_{1}.$  Indeed, we have that for $\epsilon,\delta>0,$$1/\lambda>r,$  and $\beta>0,$ 
\begin{equation}\label{eq:8a}
\begin{split}
&A_{1}(\lambda)\leq \frac{\lambda}{\overline{\Pi}_{h}(\log(1/\lambda))}\int_{(\log(\beta/\lambda), \infty)} V_{h}(dx)e^{x}\er\left(\exp\left\{-\lambda e^{x}S_{\log(r\beta)}\right\}\right)\\
&\leq \sum_{n\geq 0}\frac{\lambda}{\overline{\Pi}_{h}(\log(1/\lambda))}\int_{(n\delta+\log(\beta/\lambda), (n+1)\delta+\log(\beta/\lambda))} V_{h}(dx)e^{x}\er\left(\exp\left\{-\lambda e^{x}S_{\log(r\beta)}\right\}\right)\\
&\leq \sum_{n\geq 0}\frac{\left(\overline{V}_{h}(\log(\beta/\lambda)+n\delta)-\overline{V}_{h}(\log(\beta/\lambda)+(n+1)\delta)\right)}{\overline{\Pi}_{h}(\log(1/\lambda))}\beta e^{(n+1)\delta}\er\left(\exp\left\{-\beta e^{n\delta}S_{\log(r\beta)}\right\}\right)\\
&\leq \frac{1}{(\phi_{h}(-\alpha))^{2}}\frac{\overline{\Pi}_{h}(\log(\beta/\lambda))}{\overline{\Pi}_{h}(\log(1/\lambda))}\beta\sum_{n\geq 0}\left[\left(e^{-\alpha(n\delta)}-e^{-\alpha(n+1)\delta)}\right)+\epsilon\delta\right]e^{(n+1)\delta}\er\left(\exp\left\{-\beta e^{n\delta}S_{\log(r\beta)}\right\}\right)\\
\end{split}
\end{equation}
By the monotonicity of the exponential function, elementary arguments and making a change of variables $u=\beta e^{x-\delta},$ it follows that
\begin{equation}\label{eq:9}
\begin{split}
&\sum_{n\geq 0}\left[\left(e^{-\alpha(n\delta)}-e^{-\alpha(n+1)\delta)}\right)\right]e^{(n+1)\delta}\er\left(\exp\left\{-\beta e^{n\delta}S_{\log(r\beta)}\right\}\right)\\&\leq\sum_{n\geq0}\alpha\int^{(n+1)\delta}_{n\delta}dxe^{-\alpha x}e^{\delta}e^{x}\er\left(\exp\left\{-\beta e^{x-\delta}S_{\log(r\beta)}\right\}\right)\\
&=e^{\delta}\alpha\int^{\infty}_{0}dxe^{-\alpha x}e^{x}\er\left(\exp\left\{-\beta e^{x-\delta}S_{\log(r\beta)}\right\}\right)\\
&=\alpha\beta^{\alpha-1}e^{\delta(2-\alpha)}\int^{\infty}_{\beta e^{-\delta}}duu^{-\alpha}\er\left(\exp\left\{-uS_{\log(r\beta)}\right\}\right)
\end{split}
\end{equation} 
And similarly that
\begin{equation}\label{eq:10}
\begin{split}
\epsilon\sum_{n\geq 0}\delta e^{(n+1)\delta}\er\left(\exp\left\{-\beta e^{n\delta}S_{\log(r\beta)}\right\}\right)&\leq \epsilon \beta^{-1}e^{2\delta}\int^\infty_{\beta e^{-\delta}}du\er\left(\exp\left\{-uS_{\log(r\beta)}\right\}\right)\\
&\leq \epsilon \beta^{-1}e^{2\delta}\int^\infty_{0}du\er\left(\exp\left\{-uS_{\log(r\beta)}\right\}\right)\\ 
&= \epsilon \beta^{-1}e^{2\delta}\er\left(\frac{1}{S_{\log(r\beta)}}\right) \end{split}
\end{equation}
Using the former and latter inequalities in (\ref{eq:8a}) and making $\lambda\to 0$ we get that for $\delta, \beta, \epsilon, r>0,$
\begin{equation}
\begin{split}
\limsup_{\lambda\to 0}A_{1}(\lambda)\leq \frac{1}{\left(\phi_{h}(-\alpha)\right)^2}\alpha e^{\delta(2-\alpha)}\int^{\infty}_{\beta e^{-\delta}}duu^{-\alpha}\er\left(\exp\left\{-uS_{\log(r\beta)}\right\}\right)+\epsilon e^{2\delta}\er\left(\frac{1}{S_{\log(r\beta)}}\right).
\end{split}
\end{equation}
Now by the monotone convergence theorem and the fact that $S_{\log(r\beta)}\uparrow_{r\to\infty}I,$ $\pr$-a.s. it follows that when we make $\delta\to 0,$ then $r\to \infty,$ and finally $\beta\to 0,$ we obtain the upper bound 
\begin{equation}
\begin{split}
\limsup_{\lambda\to0}A_{1}(\lambda)\leq \frac{1}{\left(\phi_{h}(-\alpha)\right)^2}\alpha \int^{\infty}_{0}duu^{-\alpha}\er\left(\exp\left\{-uI\right\}\right)+\epsilon\er\left(\frac{1}{I}\right).
\end{split}
\end{equation}
On account of the hypothesis $-\infty<\er(\xi_{1})<0,$ we can ensure that $\er(I^{-1})<\infty,$ and hence we infer that
$$\limsup_{\lambda\to 0}A_{1}(\lambda)\leq \frac{\alpha\Gamma(1-\alpha)}{(\phi_{h}(-\alpha))^{2}}\er(I^{\alpha-1}).$$ 
An argument analogous to the one above gives also
$$\liminf_{\lambda\to 0}A_{1}(\lambda)\geq \frac{\alpha\Gamma(1-\alpha)}{(\phi_{h}(-\alpha))^{2}}\er(I^{\alpha-1}).$$
We will not reproduce the argument as the only point that needs special care is that we can choose $r$ large enough such that $\er(S^{-1}_{\log(r\beta)})<\infty$, which we now exists because $\er(S^{-1}_{\log(r\beta)})\downarrow_{r\to\infty}\er(I^{-1})<\infty.$  
 
 To finish the proof we need to prove that $A_{2}(\lambda)\xrightarrow[\lambda \to 0]{}0.$ To that end we start by observing that $A_{2}$ can be bounded by above by
\begin{equation*}
\begin{split}
\limsup_{\lambda\to 0}A_{2}(\lambda)&\leq=\frac{\lambda\int^{\log(\beta/\lambda)}_{0}V_{h}(dx)(e^x-1)+\lambda V_{h}(0,\log(\beta/\lambda)]}{\overline{\Pi}_{h}(\log(1/\lambda))}\\
&= \frac{\lambda\int^{\log(\beta/\lambda)}_{0}du e^{u}\left(\overline{V}_{h}(u)-\overline{V}_{h}(\log(\beta/\lambda))\right)+\lambda V_{h}(0,\log(\beta/\lambda)]}{\overline{\Pi}_{h}(\log(1/\lambda))}\\
&\leq \frac{\lambda\int^{\beta/\lambda}_{1}du\overline{V}_{h}(u)+\lambda V_{h}(0,\log(\beta/\lambda)]}{\overline{\Pi}_{h}(\log(1/\lambda))}\\&\sim\frac{(1+\alpha)^{-1}(\beta/\lambda)\lambda\overline{V}(\log(\beta/\lambda))\left[1+o(1)\right]}{\overline{\Pi}_{h}(\log(1/\lambda))}\\
&\sim \frac{\beta^{(1-\alpha)}}{(1+\alpha)(\phi_{h}(-\alpha))^{2}},
\end{split}
\end{equation*}
where the first equality follows by an integration by parts, the first estimate is a consequence of Karamata's Theorem (Proposition 1.5.8 in \cite{BGT}) and the fact that $\int^{\beta/\lambda}_{1}du\overline{V}_{h}(u)$ tends to infinity as $\lambda\to 0,$ and finally the last estimate follows from the fact in (\ref{U}). Therefore making $\beta\to 0,$ we obtain that $$0\leq \limsup _{\lambda\to 0}A_{2}(\lambda)\leq 0,$$ which finishes the proof.  
\end{proof}
\begin{proof}[Proof of Lemma~\ref{lemma:5}: case $\alpha=1$]
We will prove that 
\begin{equation}\label{eq=alpha1}\lim_{y\to\infty}\frac{\int^{\lambda y}_{y}\overline{\Pi}_{Y}(u)du}{y\overline{\Pi}_{h}(\log(y))}=\frac{1}{(\phi_{h}(-1))^2}\log(\lambda),\qquad \lambda >1.\end{equation} The result will follow from this on account of Theorem 3.6.8 in \cite{BGT}, which allow us to ensure that in that case $$\lim_{y\to\infty}\frac{\overline{\Pi}_{Y}(y)}{\overline{\Pi}_{h}(\log(y))}=\frac{1}{(\phi_{h}(-1))^2}.$$
To establish an upper bound for the limit in (\ref{eq=alpha1}) we observe that from Lemma \ref{lemma:1} that the numerator can be written as follows 
\begin{equation*}
\begin{split}
\int^{\lambda y}_{y}\overline{\Pi}_{Y}(u)du&=\int_{(\log(y\beta),\infty)}V_{h}(dx)e^x\pr(y<e^{x}S_{x}<\lambda y)\\
&\ +\int_{(0,\log(y\beta))}V_{h}(dx)e^x\pr(y<e^{x}S_{x}<\lambda y)\\
&:=B_{1}(y)+B_{2}(y)
\end{split}
\end{equation*}
for $\beta>0$ and $S_{x}:=\int^{T_{(-\infty,-x)}}_{0}e^{\xi_{s}}ds, \ x>0.$ The term $B_{1}$ can be bounded by above and below by 
\begin{equation}\label{eq:b1}
\begin{split}
&\int_{(\log(s\beta),\infty)}V_{h}(dx)e^x\pr(ye^{-x}< S_{\log(r\beta)}, I<\lambda ye^{-x})\leq B_{1}(y)\\
&\leq \int_{(\log(s\beta),\infty)}V_{h}(dx)e^x\pr(ye^{-x}< I, S_{\log(r\beta)}<\lambda ye^{-x}),
\end{split}
\end{equation}
for any $r<y;$ these follow from the inequalities in (\ref{Sincreasing}).  
Now, let $\delta,\epsilon>0$ and use the uniformity in (\ref{U}) to bound by below the left hand side in the equation (\ref{eq:b1}) as follows 
\begin{equation}\label{eq:15}
\begin{split}
&\frac{1}{y\overline{\Pi}_{h}(\log(y\beta))}\int_{(\log(y\beta),\infty)}V_{h}(dx)e^x\pr(ye^{-x}< S_{\log(r\beta)}, I<\lambda ye^{-x})\\
&\geq \beta\sum_{n\geq 0}\frac{\left(\overline{V}_{h}(\log(y\beta)+n\delta)-\overline{V}({\log(y\beta)+(n+1)\delta})\right)}{y\overline{\Pi}_{h}(\log(y\beta))}e^{n\delta}\pr(\beta^{-1} e^{-n\delta}< S_{\log(r\beta)}, I<\lambda \beta^{-1} e^{-(n+1)\delta})\\
&\geq \beta\sum_{n\geq 0}\left[\frac{1}{(\phi_{h}(-1))^2}\left(e^{-n\delta}-e^{-(n+1)\delta}\right)-\epsilon\delta\right]e^{n\delta}\pr(\beta^{-1} e^{-n\delta}< S_{\log(r\beta)}, I<\lambda \beta^{-1} e^{-(n+1)\delta}).
\end{split}
\end{equation}
for $y$ large enough. To study the right most term in the latter equation we argue as in (\ref{eq:9}) to get 
\begin{equation}\label{eq:16}
\begin{split}
&\sum_{n\geq 0}\left[\left(e^{-n\delta}-e^{-(n+1)\delta}\right)\right]e^{n\delta}\pr(\beta^{-1} e^{-n\delta}< S_{\log(r\beta)}, I<\lambda \beta^{-1} e^{-(n+1)\delta})\\
&\geq  e^{-\delta}\int^{\infty}_{0}dx\pr(\beta^{-1} e^{-(x-\delta)}< S_{\log(r\beta)}, I<\lambda \beta^{-1} e^{-(x+\delta)})\\
&= e^{-\delta}\int^{\beta^{-1} e^{\delta}}_{0}\frac{du}{u}\pr(u< S_{\log(r\beta)}, I<\lambda u e^{-2\delta}).
\end{split}
\end{equation} 
The rightmost term in the latter inequality is finite as $S_{\log(r\beta)}\leq I,$ $\pr$-a.s. and hence by Fubini's theorem $$\int^{\beta^{-1} e^{\delta}}_{0}\frac{du}{u}\pr(u< S_{\log(r\beta)}, I<\lambda u e^{-2\delta})\leq \int^{\infty}_{0}\frac{du}{u}\pr(I e^{2\delta}/\lambda<u<I)=2\delta+\log(\lambda).$$
Moreover arguing as in (\ref{eq:10}) we obtain that
\begin{equation}\label{eq:17}
\begin{split}
&\epsilon\sum_{n\geq 0}\delta e^{n\delta}\pr(\beta^{-1} e^{-n\delta}< S_{\log(r\beta)}, I<\lambda \beta^{-1} e^{-(n+1)\delta})\\
&\leq \epsilon e^{-\delta}\int^\infty_{0}dxe^{x}\pr(\beta^{-1} e^{-x}< S_{\log(r\beta)}, I<\lambda \beta^{-1} e^{-x})\\
&=\epsilon\beta^{-1}e^{-\delta}\int^\infty_{\beta}du\pr(S^{-1}_{\log(r\beta)}<u<\lambda I^{-1})\\
&\leq \epsilon\beta^{-1}e^{-\delta}\int^\infty_{0}du\pr(I^{-1}<u<\lambda I^{-1})= \epsilon\beta^{-1}e^{-\delta}(\lambda-1)\er(I^{-1})<\infty.
\end{split}
\end{equation}
Putting the pieces together we get 
\begin{equation*}
\begin{split}
\liminf_{y\to\infty}\frac{B_{1}(y)}{y\overline{\Pi}_{h}(\log(y))}&=\lim_{y\to\infty}\frac{\overline{\Pi}_{h}(\log(y\beta))}{\overline{\Pi}_{h}(\log(y))}\liminf_{y\to\infty}\frac{B_{1}(y)}{y\overline{\Pi}_{h}(\log(y\beta))}\\
&\geq \frac{e^{-\delta}}{(\phi_{h}(-1))^2}\int^{\beta^{-1} e^{\delta}}_{0}\frac{du}{u}\pr(u< S_{\log(r\beta)}, I<\lambda u e^{-2\delta})\\
&\ -\epsilon e^{-\delta}\er\left(I^{-1}\right) ,\qquad \delta, r, \beta>0.
\end{split}
\end{equation*}
By making $\epsilon, \delta\to 0,$ then $r\to\infty$ and finally $\beta\to 0$ and using the monotone convergence theorem we obtain the lower bound 
\begin{equation*}
\begin{split}
\liminf_{y\to\infty}\frac{B_{1}(y)}{y\overline{\Pi}_{h}(\log(y))}\geq \frac{1}{(\phi_{h}(-1))^2}\log(\lambda),\qquad \lambda>1.
\end{split}
\end{equation*}
It is proved in a similar way that 
\begin{equation*}
\begin{split}
\limsup_{y\to\infty}\frac{B_{1}(y)}{y\overline{\Pi}_{h}(\log(y))}\leq \frac{1}{(\phi_{h}(-1))^2}\log(\lambda),\qquad \lambda>1,
\end{split}
\end{equation*} we omit the details.
So, to finish the proof we need to prove that the term $B_{2}$ is $o(y\overline{\Pi}_{h}(\log(y))).$ Indeed, observe that by Fubini's theorem  and Tchebyshev's inequality it follows that $$\er(I)=\int^{\infty}_{0}\er\left(e^{\xi_{s}}\right)ds=\frac{1}{\psi(1)}<\infty,\qquad \pr(I>y)\leq \frac{\er(I)}{y}=\frac{1}{y\psi(1)},\qquad y>0.$$ 
Using this inequality and making an integration by parts in the following expression we conclude that
\begin{equation*}
\begin{split}
\int_{(0,\log(y\beta))}V_{h}(dx)e^x\pr(y<e^{x}S_{x}<\lambda y)&\leq\int_{(0,\log(y\beta))}V_{h}(dx)e^x\pr(y<e^{x}I)\\
&\leq \frac{1}{y\psi(1)}\int_{(0,\log(y\beta))}V_{h}(dx)e^{2x}\\
&= \frac{1}{y\psi(1)}\left(V_{h}(0,\log(y\beta))+\int_{(0,\log(y\beta))}V_{h}(dx)(e^{2x}-1)\right)\\
&\leq \frac{1}{y\psi(1)}\left(V_{h}(0,\infty)+2\int_{(0,\log(y\beta))}e^{2x}\overline{V}_{h}(x)dx\right).
\end{split}
\end{equation*}
Moreover, by making a change of variables $u=e^x$ and using Karamata's Theorem we get
\begin{equation*}
\begin{split}
\int_{(0,\log(y\beta))}e^{2x}\overline{V}_{h}(x)dx&=\int^{y\beta}_{1}u\overline{V}_{h}(\log(u))du\\
&\sim (y\beta)^2\overline{V}_{h}(\log(y\beta)),\qquad y\to\infty.
\end{split}
\end{equation*}
It follows from the latter and former estimates, the estimate in (\ref{renewalasym}) and the fact that $y^2\overline{\Pi}_{h}(\log(y))$ tends to $\infty$ as $y\to\infty$ that
\begin{equation}
\begin{split}
0\leq \limsup_{y\to\infty}\frac{B_{2}(y)}{y\overline{\Pi}_{h}(\log(y))}&\leq \limsup_{y\to\infty}\frac{\left(V_{h}(0,\infty)+2\int_{(0,\log(y\beta))}e^{2x}\overline{V}_{h}(x)dx\right)}{\psi(1)y^2\overline{\Pi}_{h}(\log(y))}= \frac{2\beta}{\psi(1)}. 
\end{split}
\end{equation}
Making $\beta$ tend to $0$ we get the claimed result.
\end{proof}
\begin{proof}[Proof of Lemma~\ref{lemma:5}: case $\alpha>1$]
Observe that by the hypotheses [H4], the integral $\int^{\infty}_{0}V_{h}(dx)e^{x}<\infty,$ (see e.g. \cite{KKM} Proposition 4.2). In account of the monotone density theorem (Theorem 1.7.2 in \cite{BGT}) for regularly varying functions it is enough to prove that 
$$y\mapsto\int^{\infty}_{y}\overline{\Pi}_{Y}(u)du=\int_{\re_{+}}V_{h}(dx)e^{x}\pr\left(e^{x}\int^{T_{(-\infty,-x)}}_{0}e^{\xi_{s}}ds> y\right),\qquad y\geq 0,$$ is regularly varying at infinity with index $-(1+\alpha).$ As in the proof of the case $\alpha<1$ we will compare the latter quantity to $y\overline{\Pi}_{h}(\log(y)),$ at infinity. Indeed, an application of Fubini's theorem leads 
\begin{equation}\label{eq:19}
\begin{split}
\int_{(0,\infty)}V_{h}(dx)e^{x}\pr\left(e^{x}\int^{T_{(-\infty,-x)}}_{0}e^{\xi_{s}}ds> y\right)&=\er\left(\int_{(0,\infty)}V_{h}(dx)e^{x}1_{\left\{x>\log\left(y/S_{x}\right)\right\}}\right)\\
&\leq \er\left(\int_{(0,\infty)}V_{h}(dx)e^{x}1_{\left\{x>\log\left(y/I\right)\right\}}1_{\{I^{-1}>\delta \}}\right)\\
&+ \er\left(\int_{(\log(y\delta ),\infty)}V_{h}(dx)e^{x}1_{\left\{x>\log\left(y/I\right)\right\}}1_{\{I^{-1}\leq \delta \}}\right)\\
&+ \er\left(\int_{(0,\log(y\delta )]}V_{h}(dx)e^{x}1_{\left\{x>\log\left(y/I\right)\right\}}1_{\{I^{-1}\leq \delta \}}\right)\\
&=:\er\left(F(y/I)1_{\{I^{-1}>\delta \}}\right)+C_{1}(y)+C_{2}(y)
\end{split}
\end{equation}
for $y, \delta>0,$ where $F(z):=\int_{(0,\infty)} V_{h}(dx)e^x1_{\{x>\log(z)\}},$ $z>0.$ 
To study the first term on the rightmost term in equation (\ref{eq:19}) we claim that $F$ is a function which is regularly varying with an index $1-\alpha$ at infinity and such that
\begin{equation}\label{eq:21}
\lim_{z\to\infty}\frac{F(z)}{z\overline{V}_{h}(\log(z))}=\frac{\alpha}{\alpha-1}.
\end{equation}  
Furthermore, on account of the regular variation of $F$ with a negative index $1-\alpha,$ we have that for $v>0,$ $$\lim_{z\to\infty}\frac{F(z\lambda)}{F(z)}=\lambda^{1-\alpha}, \qquad \text{uniformly in}\ \lambda\in(v,\infty),$$ see e.g. chapter 1 in \cite{BGT}.
Let us prove that the limit in equation (\ref{eq:21}) hold; the regular variation of $F$ follows therefrom. Using Fubini's theorem and Karamata's Theorem  we get that
\begin{equation*}
\begin{split}
F(z)&=z\overline{V}_{h}(\log(z))+\int^\infty_{\log(z)}V_{h}(dx)\left(e^{x}-e^{\log(z)}\right)\\
&=z\overline{V}_{h}(\log(z))+\int^\infty_{\log(z)}dxe^{x}\overline{V}_{h}(x)\\
&=z\overline{V}_{h}(\log(z))+\int^\infty_{z}du\overline{V}_{h}(\log(u))\\
&\sim z\overline{V}_{h}(\log(z))\left(1+\frac{1}{\alpha-1}\right),\qquad z\to\infty,
\end{split}
\end{equation*} which proves (\ref{eq:21}).
From the properties of $F$ and the estimate (\ref{renewalasym}) we infer that for $\delta>0,$
\begin{equation}\label{eq:22}\lim_{y\to\infty}\frac{\er\left(F(y/I)1_{\{I^{-1}>\delta\}}\right)}{y\overline{\Pi}_{h}(\log(y))}=\frac{1}{(\phi_{h}(-\alpha))^2}\frac{\alpha}{\alpha-1}\er\left(I^{\alpha-1}1_{\{I^{-1}>\delta\}}\right).\end{equation}
We next study the terms $C_{1}$ and $C_{2}.$ The term $C_{1}$ can be bounded by above as follows
\begin{equation*}\label{eq:23}
\begin{split}
C_{1}(y)&\leq\pr(I^{-1}\leq \delta)\int^{\infty}_{\log(y\delta)}V_{h}(dx)e^{x}\\
&\leq \delta^{\alpha}\er(I^{\alpha})\left(y\delta\overline{V}_{h}(\log(y\delta))+\int^\infty_{y\delta}du\overline{V}_{h}(\log(u))\right)\\
&\sim \frac{\alpha \delta^{\alpha}\er(I^{\alpha})}{\alpha-1}y\delta\overline{V}_{h}(\log(y\delta)),\qquad y\to\infty.
\end{split}
\end{equation*}where the first inequality follows from the very definition of $C_{1},$ the second inequality from Tchebyshev's inequality and an integration by parts and finally the estimate follows from Karamata's Theorem. Using the estimate (\ref{renewalasym}) and the regular variation of $\overline{V}_{h}(\log(\cdot))$ we conclude that
\begin{equation}\label{eq:24}0\leq \limsup\frac{C_{1}(y)}{y\overline{\Pi}_{h}(\log(y))}\leq  \delta\frac{\alpha^2\er(I^{\alpha})}{(\alpha-1)(-\phi_{h}(-\alpha))^2}.\end{equation}
We now need to determine the rate of growth of $C_{2}.$ This can be done as follows
\begin{equation*}\label{eq:25}
\begin{split}
C_{2}(y)&\leq \int^{\log(y\delta)}_{0}{V}_{h}(dx)e^{x}\pr(I>ye^{-x})\\
&\leq y^{-\alpha}\er(I^{\alpha})\int^{\log(y\delta)}_{0}{V}_{h}(dx)e^{(1+\alpha)x}\\
&= y^{-\alpha}\er(I^{\alpha})\left(V_{h}(0,\log(y\delta))+\int^{y\delta}_{1}du u^{\alpha}\overline{V}_{h}(\log(u))\right)\\
&\sim \er(I^{\alpha})y^{-\alpha}(\delta y)^{\alpha+1}\overline{V}(\log(\delta y)),
\end{split}
\end{equation*}
where the first inequality follows from the definition of $C_{2},$ the second from an application of Tchebyshev's inequality, the equality follows by an integration by parts and a change of variables and finally the estimate follows from Karamata's Theorem  and the fact that $V_{h}$ is a finite measure. We infer therefrom using the estimate (\ref{renewalasym}) and the regular variation of $\overline{V}_{h}(\log(\cdot))$ that
\begin{equation}\label{eq:26}
\limsup_{y\to\infty}\frac{C_{2}(y)}{y\overline{\Pi}_{h}(\log(y))}\leq \delta\frac{\er(I^{\alpha})}{(-\phi_{h}(-\alpha))^2}
\end{equation}
Plugging the estimates in (\ref{eq:22}), (\ref{eq:24}), (\ref{eq:26}) in the inequality (\ref{eq:19}) we conclude that 
\begin{equation}\label{eq:27}
\limsup_{y\to\infty}\frac{\int^{\infty}_{y}\overline{\Pi}_{Y}(u)du}{y\overline{\Pi}_{h}(\log(y))}\leq \frac{1}{(\phi_{h}(-\alpha))^2}\frac{\alpha}{\alpha-1}\er\left(I^{\alpha-1}1_{\{I^{-1}>\delta\}}\right)+\delta\gamma_{\alpha},\qquad \text{for} \ \delta>0,
\end{equation}
where $\gamma_{\alpha}$ is a positive and finite constant whose value is the addition of the constants appearing in (\ref{eq:24}) and (\ref{eq:26}). Making $\delta$ tend to $0$ we get the upper bound
\begin{equation}\label{eq:28}
\limsup_{y\to\infty}\frac{\int^{\infty}_{y}\overline{\Pi}_{Y}(u)du}{y\overline{\Pi}_{h}(\log(y))}\leq \frac{1}{(\phi_{h}(-\alpha))^2}\frac{\alpha}{\alpha-1}\er\left(I^{\alpha-1}\right).
\end{equation}
 To obtain a lower bound we use the inequality
\begin{equation*}\label{eq:20}
\begin{split}
\frac{\int_{(0,\infty)}V_{h}(dx)e^{x}\pr\left(e^{x}\int^{T_{(-\infty,-x)}}_{0}e^{\xi_{s}}ds> y\right)}{y\overline{\Pi}_{h}(\log(y))}&\geq \frac{\er\left(\int_{(r,\infty)}V_{h}(dx)e^{x}1_{\left\{x>\log\left(y/S_{r}\right)\right\}}1_{\{S^{-1}_{r}>\delta\}}\right)}{y\overline{\Pi}_{h}(\log(y))}\\
&=\frac{\er\left(F(y/S_{r})1_{\{S^{-1}_{r}>\delta\}}\right)}{y\overline{\Pi}_{h}(\log(y))}
\end{split}
\end{equation*}
for $\delta>0$ and $\log(y\delta)>r>0.$ Using this inequality and arguing as for (\ref{eq:22}) we get that  
\begin{equation*}
\liminf_{y\to\infty}\frac{\int^\infty_{y}\overline{\Pi}_{Y}(z)dz}{y\overline{\Pi}_{h}(\log(y))}\geq \frac{1}{(\phi_{h}(-\alpha))^2}\frac{\alpha}{\alpha-1}\er\left(S_{r}^{\alpha-1}1_{\{S^{-1}_{r}>\delta\}}\right),\qquad \text{for} \ \delta>0, r>0.
\end{equation*}
Making $a$ tend to $0$ and then $r$ towards $\infty,$ using the monotone convergence theorem and using the estimate in equation (\ref{eq:28}) we conclude that 
\begin{equation*}
\lim_{y\to\infty}\frac{\int^\infty_{y}\overline{\Pi}_{Y}(z)dz}{y\overline{\Pi}_{h}(\log(y))}= \frac{1}{(\phi_{h}(-\alpha))^2}\frac{\alpha}{\alpha-1}\er\left(I^{\alpha-1}\right).
\end{equation*}
The conclusion of the Lemma follows therefrom using the monotone density theorem for regularly varying functions, Theorem 1.7.2 in \cite{BGT}.
\end{proof}
\section{Proof of Theorem~\ref{teo3}}\label{sect4}
\subsection{Proof of assertion (i) }
The assumptions (MZ1-2) imply that $\er(\widehat{h}_{1})=\mu_{\widehat{h}}<\infty,$ (Corollary 4 Section 4.4 in \cite{doneybook}) and together with Theorem 3-(a) in \cite{R07b}  imply that 
\begin{equation}\label{eq:35c}
\overline{\Pi}_{h}(y)\sim\frac{1}{\mu_{\widehat{h}}}\int^{\infty}_{y}\overline{\Pi}^{+}(x)dx,\qquad y\to \infty.
\end{equation} Hence, by asymptotic equivalence, $\Pi_{h}\in \mathcal{S}_{0}$ (Lemma 2.4 and Corollary 2.1 in \cite{pakes2004} or Lemma 2.1 in \cite{watanabe}) and so we can apply the results in \cite{KKM} to ensure that (\ref{renewalasym}) holds with $\alpha=0$. So, to prove our claim it will be enough to prove that 
$$\lim_{y\to\infty}\frac{\overline{\Pi}_{Y}(y)}{\overline{\Pi}_{h}(\log(y))}=0.$$
To that end we start by proving that 
\begin{equation}\label{eq:34a}
\lim_{z\to\infty}\frac{\int^{\log(z)}_{0}V_{h}(dx)e^{x}}{z\overline{V}_{h}(\log(z))}=0.
\end{equation}
Indeed, by an application of Fubini's theorem and a change of variables we get the identity
\begin{equation*}\label{eq:35}
\begin{split}
\int^{z}_{0}V_{h}(dx)e^{x}&=V_{h}(0,z)+\int^{z}_{0}V_{h}(dx)(e^{x}-1)\\
&=V_{h}(0,z)+\int^{z}_{0}due^{u}\left(\overline{V}_{h}(u)-\overline{V}_{h}(z)\right)\\
&=V_{h}(0,\infty)-e^{z}\overline{V}_{h}(z)+\int^{e^{z}}_{1}ds\overline{V}_{h}(\log(s)),
\end{split}
\end{equation*} for $z>0.$
On account of the fact that under the present assumptions $\overline{V}_{h}(\log(\cdot))$ is slowly varying, we  can apply Karamata's Theorem  to get that $$\int^{y}_{1}ds\overline{V}_{h}(\log(s))\sim y\overline{V}_{h}(\log(y)).$$ Also by properties of slowly varying function we have that $y\overline{V}_{h}(\log(y))\to\infty$ as $y\to\infty.$ Putting the pieces together we conclude that the assertion in (\ref{eq:34a}) holds true. 

So by Lemma~\ref{lemma:1}, Fubini's theorem and elementary manipulations we have that 
\begin{equation}\label{eq:36}
\begin{split}
\int^{y}_{0}\overline{\Pi}_{Y}(x)dx&=\int_{(0,\infty)}V_{h}(dx)e^{x}\pr\left(e^{x}S_{x}\leq y\right)\\
&\leq \int_{(0,\log(y))}V_{h}(dx)e^{x}+\int_{(\log(y),\infty)}V_{h}(dx)e^{x}\pr\left(e^{x}S_{r}\leq y\right)\\
&\leq o(y\overline{V}_{h}(\log(y)))+\int_{(0,\infty)}V_{h}(dx)e^{x}\pr\left(e^{x}S_{r}\leq y\right)\\
&\leq o(y\overline{V}_{h}(\log(y)))+\er\left(\int_{(0,\log(y/S_{r}))}V_{h}(dx)e^{x}1_{\{S^{-1}_{r}\leq \beta\}}\right)\\&\ +\er\left(\int_{(0,\log(y\beta))}V_{h}(dx)e^{x}1_{\{S^{-1}_{r}>\beta\}}\right)+\er\left(\int_{(\log(y\beta),\log(y/S_{r}))}V_{h}(dx)e^{x}1_{\{S^{-1}_{r}>\beta\}}\right),
\end{split}
\end{equation} 
for $\beta>0,$ $0<r<\log(y),$ and $y$ large enough. We assume that $r$ is such that $\er(S^{-1}_{r})<\infty,$ which is possible as $S_{r}\uparrow I$ as $r\to\infty$ and by the monotone convergence theorem $\lim_{r\to\infty}\er\left(S^{-1}_{r}\right)=\er(I^{-1})<\infty,$ by the hypothesis that $\er(-\xi_{1})\in(0,\infty).$ The terms on the rightmost hand side in equation~(\ref{eq:36}) can be analyzed as follows: 
\begin{equation*}
\begin{split}
&\er\left(\int_{(0,\log(y/S_{r}))}V_{h}(dx)e^{x}1_{\{S^{-1}_{r}\leq \beta\}}\right)+\er\left(\int_{(0,\log(y\beta))}V_{h}(dx)e^{x}1_{\{S^{-1}_{r}>\beta\}}\right)\\
&\leq 2\int_{(0,\log(y\beta))}V_{h}(dx)e^{x}=o(y\overline{V}_{h}(\log(y))),
\end{split}
\end{equation*}
and  
\begin{equation*}
\begin{split}
\er\left(\int_{(\log(y\beta),\log(y/S_{r}))}V_{h}(dx)e^{x}1_{\{S^{-1}_{r}>\beta\}}\right)\leq \er\left(\frac{1}{S_{r}}1_{\{S^{-1}_{r}>\beta\}}\right)y\overline{V}_{h}(\log(y\beta)).
\end{split}
\end{equation*}
We deduce therefrom that $$\limsup_{y\to\infty}\frac{\int^{y}_{0}\overline{\Pi}_{Y}(x)dx}{y\overline{V}_{h}(\log(y))}\leq \er\left(\frac{1}{S_{r}}1_{\{S^{-1}_{r}>\beta\}}\right),\qquad \beta>0.$$ Making $\beta\to\infty$ we infer that  
$$\lim_{y\to\infty}\frac{\int^{y}_{0}\overline{\Pi}_{Y}(x)dx}{y\overline{V}_{h}(\log(y))}=0.$$ It follows from (\ref{renewalasym}) and (\ref{eq:35c}) that
$$\lim_{y\to\infty}\frac{\int^{y}_{0}\overline{\Pi}_{Y}(x)dx}{y\int^{\infty}_{\log(y)}\overline{\Pi}^{+}(x)dx}=0.$$ The claim follows from the monotone density theorem for regularly varying functions.
\subsection{Proof of assertion (ii)}
The proof in this part is quite similar, but simpler, to that of Lemma~\ref{lemma:5}. The main difference is that here we will use the renewal theorem instead of the results in~\cite{KKM}, which we recall were crucial in our development. Hence we will just outline the main steps of the proof. 

Assume that $\xi$ satisfies the hypotheses in (R1-3). This implies that $h$ is not arithmetic, 
$$\er(e^{\theta h_{1}})=1,\ \text{and}\ \mu^{(\theta)}_{h}:=\er(h_{1}e^{\theta h_{1}})<\infty.$$ The reason for this is the extended form of the Wiener-Hopf factorization stated in (\ref{WHextd}), as under this assumptions $[0,\theta]\subseteq C.$ Indeed, it implies that $$0=-\psi(\theta)=-\phi_{h}(-\theta)\phi_{\widehat{h}}(\theta), \qquad \phi_{\widehat{h}}(\theta)>0,\quad \Rightarrow \phi_{h}(-\theta)=0,$$ or equivalently $\er(e^{\theta h_{1}})=1.$
Moreover, by standard arguments 
$$\infty>\er(\xi_{1}e^{\theta\xi_{1}})=\lim_{\lambda\to\theta-}\frac{\psi(\lambda)}{\theta-\lambda}= \phi_{\widehat{h}}(\theta)\lim_{\lambda\to\theta-}\frac{-\phi_{h}(-\lambda)}{\theta-\lambda}=\phi_{\widehat{h}}(\theta)\er(h_{1}e^{\theta h_{1}}).$$ That the constant $\er(I^{\theta-1})$ is finite is proved in Lemma 2 in \cite{R07} whenever $0<\theta<1;$ whilst if $\theta>1$ it follows from Lemma~\ref{lemma:30} here, because by the strict convexity of $\psi$ we have that $\er\left(e^{(\theta-1)\xi_{1}}\right)=e^{\psi(\theta-1)}<1.$

\textit{Case $\theta<1$}. From Lemma~\ref{lemma:1} we know that the tail L\'evy measure of $Y$ is such that 
$$\int^{y}_{0}\overline{\Pi}_{Y}(x)dx=\int_{\re_{+}}V_{h}(dx)e^{x}\pr\left(e^{x}\int^{T_{(-\infty,-x)}}_{0}e^{\xi_{s}}ds\leq y\right),\qquad y\geq 0.$$ It is easily verified that the function $u\mapsto e^{(1-\theta)u}\pr\left( I \leq e^{-u}\right),$ $u\in \re,$ is directly Riemman integrable, because it is an integrable function which is the product of the exponential function and a decreasing function. The renewal theorem applied to the renewal measure $V^{*}_{h}(dx):=V_{h}(dx)e^{\theta x},$ $x\geq 0,$ implies that
\begin{equation}
\begin{split}
e^{-(1-\theta)y}\int^{e^{y}}_{0}\overline{\Pi}_{Y}(x)dx&=e^{-(1-\theta)y}\int_{\re_{+}}V_{h}(dx)e^{x}\pr\left(e^{x}\int^{T_{(-\infty,-x)}}_{0}e^{\xi_{s}}ds\leq e^{y}\right)\\
&=\int_{\re_{+}}V^{*}_{h}(dx)e^{(1-\theta)(x-y)}\pr\left(\int^{T_{(-\infty,-x)}}_{0}e^{\xi_{s}}ds\leq e^{-(x-y)}\right)\\
&=\int_{\re}V^{*}_{h}(dz+y)e^{(1-\theta)z}\pr\left(\int^{T_{(-\infty,-(z+y))}}_{0}e^{\xi_{s}}ds\leq e^{-z}\right)1_{\{z>-y\}}\\
&\xrightarrow[y\to\infty]{} \frac{1}{\mu^{(\theta)}_{h}}\int_{\re}due^{(1-\theta)u}\pr\left( I \leq e^{-u}\right)\\
&=\frac{1}{\mu^{(\theta)}_{h}}\int_{\re_{+}}dz z^{-(2-\theta)}\pr\left( I \leq z\right)\\
&=\frac{1}{\mu^{(\theta)}_{h}(1-\theta)}\er(I^{\theta-1}).
\end{split}
\end{equation}
The result follows from the monotone density theorem for regularly varying functions.

\textit{Case $\theta>1$.} Arguing as in the case $\theta<1$ we get
\begin{equation}
\begin{split}
e^{(\theta-1)y}\int^{\infty}_{e^{y}}\overline{\Pi}_{Y}(x)dx&=e^{(\theta-1)y}\int_{\re_{+}}V_{h}(dx)e^{x}\pr\left(e^{x}\int^{T_{(-\infty,-x)}}_{0}e^{\xi_{s}}ds> e^{y}\right)\\
&=\int_{\re_{+}}V^{*}_{h}(dx)e^{-(\theta-1)(x-y)}\pr\left(\int^{T_{(-\infty,-x)}}_{0}e^{\xi_{s}}ds > e^{-(x-y)}\right)\\
&=\int_{\re}V^{*}_{h}(dz+y)e^{-(\theta-1)z}\pr\left(\int^{T_{(-\infty,-(z+y))}}_{0}e^{\xi_{s}}ds> e^{-z}\right)1_{\{z>-y\}}\\
&\xrightarrow[y\to\infty]{} \frac{1}{\mu^{(\theta)}_{h}}\int_{\re}due^{-(\theta-1)u}\pr\left( I > e^{-u}\right)\\
&\frac{1}{\mu^{(\theta)}_{h}}\int_{\re}due^{(\theta-1)u}\pr\left( I > e^{u}\right)\\
&=\frac{1}{\mu^{(\theta)}_{h}}\int_{\re_{+}}dz z^{\theta-2}\pr\left( I > z\right)\\
&=\frac{1}{\mu^{(\theta)}_{h}(\theta-1)}\er(I^{\theta-1}).
\end{split}
\end{equation}
The result follows using the monotone density theorem for regularly varying functions.

\textit{Case $\theta=1$.} We proceed as in its analogue in the proof of Lemma~\ref{lemma:5} by discretizing the integral. Observe that we do not need the hypothesis that $\er(I^{-1})<\infty$ because for a renewal measure $U$ which satisfies the hypothesis of the Renewal Theorem, we can assume that for any $\delta,$ and any $\epsilon$ given, there exists a $t_{0}$ such that $$(1-\epsilon)\frac{\delta}{m}\leq U(t,t+\delta]\leq (1+\epsilon)\frac{\delta}{m},\qquad t\geq t_{0},$$ where $m$ denotes the mean of the inter-arrival distribution. This fact is used instead of the uniformity property (\ref{U}). Doing so we get that for any $0<\lambda<1,$ 
\begin{equation}
\begin{split}
\int^{y}_{\lambda y}\overline{\Pi}_{Y}(x)dx&=\int_{\re_{+}}V_{h}(dx)e^{x}\pr\left(\int^{T_{(-\infty,-x)}}_{0}e^{\xi_{s}}ds> e^{-(x-y)}\geq \lambda\int^{T_{(-\infty,-x)}}_{0}e^{\xi_{s}}ds\right)\\
&=\int_{\re_{+}}V^{*}_{h}(dx)\pr\left(\int^{T_{(-\infty,-x)}}_{0}e^{\xi_{s}}ds> e^{-(x-y)}\geq \lambda\int^{T_{(-\infty,-x)}}_{0}e^{\xi_{s}}ds\right)\\
&\xrightarrow[y\to\infty]{}\frac{1}{\mu^{\theta}_{h}}\int_{\re}dz\pr(I>e^{-z}\geq \lambda I)=\frac{1}{\mu^{\theta}_{h}}\log(1/\lambda).
\end{split}
\end{equation}
The result follows from Theorem 3.6.8 in \cite{BGT}.

{\gracias} I would like to thank Antonio Murillo for insightful discussions about the topic.

\end{document}